\documentclass[a4paper,11pt,centertags,psamsfonts]{amsart}
\usepackage{amssymb}
\usepackage{times}
\usepackage[mathcal]{euscript}
\newtheorem{thm}{Theorem}
\newtheorem{lemma}{Lemma}

\newtheorem{prop}{Proposition}

\theoremstyle{definition}
\newtheorem{definition}{Definition}

\theoremstyle{remark}
\newtheorem*{rem*}{Remark}
\newtheorem{rem}{Remark}

\newcommand{\mr}{{\mathbb R}}

\newcommand{\mn}{{\mathbb N}}
\newcommand{\mz}{{\mathbb Z}}
\newcommand{\mc}{{\mathbb C}}

\renewcommand{\rho}{\varrho}
\newcommand{\eps}{\varepsilon}
\newcommand{\vphi}{\varphi}
\renewcommand{\Im}{\operatorname{Im}}
\renewcommand{\Re}{\operatorname{Re}}

\newcommand{\supp}{\operatorname{supp}}
\newcommand{\dist}{\operatorname{dist}}

\begin{document}

\title{On the distribution of eigenvalues of non-selfadjoint operators}

\author[M.Demuth, M.Hansmann, G.Katriel]{Michael Demuth, Marcel Hansmann, Guy Katriel}
\thanks{The third author is partially supported by the Humboldt Foundation (Germany).}

\address{TU Clausthal, \hspace{2pt} Institute of Mathematics,\hspace{2pt} Erzstr.1, \hspace{2pt} 38678 \hspace{2pt} Clausthal-Zellerfeld, Germany}
\email{demuth@math.tu-clausthal.de, hansmann@math.tu-clausthal.de, haggaika@yahoo.com}

\begin{abstract}
We prove quantitative bounds on the eigenvalues of non-selfadjoint
bounded and unbounded operators. We use the perturbation determinant
to reduce the problem to one of studying the zeroes of a
holomorphic function.
\end{abstract}

\maketitle

\section{Introduction}
The aim of this work is to obtain some quantitative results on the
structure of the discrete spectrum of wide classes of
non-selfadjoint linear operators. While the study of the eigenvalues of
selfadjoint operators is well-developed, much less is known in the
non-selfadjoint case, since many of the methods used to study the
discrete spectrum of selfadjoint operators, employing the
variational characterization of the eigenvalues and the fact that
the eigenvalues are real, do not apply in
the non-selfadjoint setting.\\

Our approach is based on constructing a holomorphic
function, in terms of perturbation determinants, whose zeroes are the
eigenvalues of the operator we are interested in, and using complex
analysis to obtain information on these zeroes, which in turn
translates into information on the eigenvalues. Variants of this
approach were used previously, e.g. in \cite{borichev,demuth}.\\

We develop results in the bounded and unbounded settings, each of
which is useful in applications to concrete operators.\\

In Section \ref{sec:bounded}, we assume that $A,B$ are
bounded linear operators in a complex Hilbert space, where $A$ is
selfadjoint with $\sigma(A)=[a,b]$  and $B-A$ belongs to the
Schatten class $\mathbf{S}_p, p>0$. The aim is to obtain
quantitative bounds on the discrete spectrum $\sigma_{disc}(B)$,
i.e. on isolated eigenvalues of finite algebraic multiplicity, in
terms of the $p$-Schatten norm of $B-A$. By the above assumptions,
the essential spectrum $\sigma_{ess}(B)=[a,b]$, and
$\sigma_{disc}(B)=\sigma(B)\cap (\mc \setminus [a,b])$ consists of a
sequence of eigenvalues which can only accumulate on $[a,b]$. Our
results quantify the rate at which this approach to the essential
spectrum must occur under the above assumptions. We prove that, for $\gamma>\max(1+p,2p)$,
\begin{equation}\label{first}
 \sum_{\lambda \in \sigma_{disc}(B)}
 \frac{\dist(\lambda,[a,b])^\gamma}{|\lambda-a|^{\frac \gamma 2}|\lambda-b|^{\frac \gamma 2 }}
 \leq  C \| B-A \|_{\mathbf{S}_p}^p,
\end{equation}
where the constant $C$, which is given explicitly, depends only on
$p$ and on $\gamma$.\\

In Section \ref{sec:unboundedI}, we assume that $H_0$, $H$
are (unbounded) closed operators in a complex Hilbert space,
$H_0$ is selfadjoint with $\sigma(H_0)=[0,\infty)$ and the
resolvent difference $R_s=(s-H)^{-1}-(s-H_0)^{-1}$ is in
$\mathbf{S}_p$ for some $s<0$. The eigenvalues of $H$  may
accumulate on $\sigma_{ess}(H)=[0,\infty)$, or at infinity. We
obtain quantitative information on the rate of accumulation given by
the following inequality: for $\gamma>\max(1+p,2p)$,
\begin{equation}\label{second}
\sum_{\lambda \in \sigma_{disc}(H)} \frac{
\dist(\lambda,[0,\infty))^\gamma }{|\lambda|^{\frac \gamma 2}
(1+|\lambda|)^{\gamma}}  \leq C\|R_s\|_{\mathbf{S}_p}^p,
\end{equation}
 where the constant $C$ depends only on $p,\gamma,$ and $s$.\\

The results for unbounded operators above are proved by reduction to
the case of bounded operators, so that the results in the bounded
case are fundamental. Our results for bounded operators are in the
same spirit as the results of Borichev, Golinskii, and Kupin
\cite{borichev}. A chief difference is that while in the proof of
their results, the above authors used a new result in complex
analysis about zeroes of certain holomorphic functions, whose proof
is quite involved, we use a result which is directly derived from
Jensen's identity (see Section \ref{sec:jensen}). As will be discussed
in Remark \ref{rem:rem}, our complex analysis result is sometimes weaker
than that of \cite{borichev}, but has the advantage
that it enables us to derive explicit expressions for the constants
involved. This allows us to derive the explicit form of the constant $C$
in (\ref{first}) and (\ref{second}), and their dependence on the parameters
$\gamma$ and $p$.\\

Some results for unbounded operators $H$, in the case that $H$
is selfadjoint and lower semi-bounded, were obtained previously in \cite{demuth}, by the
perturbation-determinant method. These bounds were given in terms of
Schatten norms of the semigroup difference $e^{-tH}-e^{-tH_0}$. In
\cite{hansmann} it was shown that inequalities {\it{stronger}} than
those obtained in \cite{demuth} can be proven by a completely
different method, based on the spectral shift function (and thus
restricted to selfadjoint operators). Here we see that the
perturbation-determinant method finds its natural place in the study
of non-selfadjoint operators. In this work we develop results in
terms of the resolvent difference rather than the semigroup
difference, but we note that it is also possible to derive results
in terms
of semigroup differences, using the same procedure of reduction to the bounded case.\\

In the final section of this paper, we construct a counterexample
which demonstrates that our results are sharp in a certain sense.

\section{Preliminaries}\label{sec:prelim}
For a seperable complex Hilbert space $\mathcal{H}$ let  $\mathbf{C}(\mathcal{H})$ and $\mathbf{B}(\mathcal{H})$ denote the closed and bounded linear operators on $\mathcal{H}$ respectively. We denote the ideal of all compact operators on $\mathcal{H}$ by $\mathbf{S}_\infty$ and the ideal of all Schatten class operators by $\mathbf{S}_p, p > 0$, i.e. a compact operator $C \in \mathbf{S}_p$ if
\begin{equation}
  \| C\|_{\mathbf{S}_p}^p = \sum_{n=0}^\infty \mu_n(C)^p < \infty
\end{equation}
where $\mu_n(C)$ denotes the $n$-th singular value of $C$. Suppose
that $A, B \in \mathbf{B}(\mathcal{H})$
 and $B-A \in \mathbf{S}_p$ for some real $p>0$. Since $B-A \in \mathbf{S}_{\lceil p \rceil}$, where
\[ \lceil p \rceil = \min \{ n \in \mn :  n \geq p \},\] the  $\lceil p \rceil$-regularized
perturbation-determinant of $A$ by $B-A$ is well defined as
\begin{equation}\label{eq:per_det}
 h_{A,B}^{(p)}(z)
={\det}_{\lceil p \rceil}(I-(z-A)^{-1}(B-A)),
\end{equation}
and is analytic on $\rho(A)$, the resolvent set of $A$. Furthermore,
$z_0 \in \rho(A)$ is an eigenvalue of $B$ of algebraic multiplicity
$k_0$ if and only if $z_0$ is a zero of $h_{A,B}^{(p)}$ of the same
multiplicity.
 For more information on  regularized perturbation-determinants we refer to the books by Dunford and Schwartz
 \cite{Dunford}, Gohberg and Krein \cite{Gohberg} or Simon \cite{simonb}. Note that
\begin{equation}\label{eq:infty}
\lim_{z \to \infty} h_{A,B}^{(p)}(z)=1
\end{equation} and
\begin{equation}
  \label{eq:estimate_dunford}
  |h_{A,B}^{(p)}(z)| \leq \exp \left( \Gamma_p \|(z-A)^{-1}(B-A) \|_{\mathbf{S}_p}^p\right),
\end{equation}
where $\Gamma_p $ is some positive constant, see \cite[page 1106]{Dunford}. We remark that
$\Gamma_p= \frac{1}{p}$ for $p \leq 1$, which is a direct consequence of the definition
of the determinant, $\Gamma_2=\frac 1 2$ and $\Gamma_p \leq e(2+ \log p)$ in general, see \cite[Simon]{Simon77}.
\begin{rem}\label{rem:dav}
Assuming $B-A \in \mathbf{S}_\infty$, we have
$\sigma_{ess}(A)=\sigma_{ess}(B)$. Let $G$ be the
unbounded component of $\mc \setminus \sigma_{ess}(A)$. Then $G \cap
\sigma(B) \subset \sigma_{disc}(B)$ and all possible limit points of
this set lie in $\sigma_{ess}(A)$. Here the discrete spectrum
$\sigma_{disc}(B)$ consists of those eigenvalues of $B$ that
 have finite algebraic multiplicity and are isolated from the rest of the spectrum.
 As a general reference for the mentioned definitions and results we refer to the
 book of Davies \cite[Chapter 4.3]{Davies07}, see Theorem 4.3.18 in particular.
\end{rem}
\section{Some complex analysis results}\label{sec:jensen}
We have seen that the zeroes of the analytic function
$h_{A,B}^{(p)}$ play an important role in the study of the
eigenvalues of $B$ in $\mc \setminus \sigma(A)$. In the following,
we use Jensen's identity to obtain some general results on the
distribution of zeroes of functions analytic in the open unit disk
$U$.
\begin{lemma}\label{lem:compl1}
  Let $\vphi \in C^2(0,1)$ be a nonnegative, nonincreasing function with
  \begin{equation}
\lim_{r \to 1} \vphi(r)= \lim_{r \to 1} \vphi'(r)=0
  \end{equation}
that obeys
\begin{equation} \label{eq:rv1}
  \supp \left( [ r\vphi'(r) ]' \right)_- \subset [0,1)
\quad \text{and} \quad \sup_{r \in (0,1)} \left( [ r\vphi'(r)
]' \right)_- < \infty,
\end{equation}
where $f_- = -\min(f,0)$ is the negative part of a function $f$. Let $h: U \to \mc$ be an analytic function with $h(0)=1$. Then
\begin{eqnarray}
  \sum_{z \in U,\; h(z)=0} \vphi(|z|) &=& \frac 1 {2\pi} \int_0^1 dr \:
  [ r\vphi'(r) ]'   \int_0^{2\pi} d\theta \log |h(re^{i\theta})| \label{eq:compl1}
\end{eqnarray}
where in the sum each zero of $h$ is counted according to its
multiplicity.
\end{lemma}
\begin{rem}
A short calculation shows that $\vphi_1(r)=|\log(r)|^\gamma$, $\vphi_2(r)=(1-r)^\gamma$ and $\vphi_3(r)=(r^{-1}-r)^\gamma$ fulfill the above assumptions in case that $\gamma > 1$.
\end{rem}
\begin{proof}
Jensen's identity states that
\begin{eqnarray}\label{eq:Jensen}
0 \leq \int_0^r ds \: s^{-1} n(h;U_s) = \frac{1}{2\pi}\int_0^{2\pi} d\theta \: \log| h(re^{i\theta})|, \quad 0< r < 1
\end{eqnarray}
where $n(h ; U_s)$ counts the number of zeroes of $h$ (including
multiplicities) in the disk of radius $s$, see e.g. Rudin \cite[page
$307$]{Rudin}. Multiplying both sides of (\ref{eq:Jensen}) by $[
r\vphi'(r) ]' $ and integrating over $r \in [0,1]$ leads to
 \begin{eqnarray}
&&  \frac{1}{2\pi} \int_0^1 dr\: [ r\vphi'(r) ]'  \int_0^{2\pi} d \theta\: \log | h(re^{i\theta})|  \nonumber\\
&=& \int_0^1 dr \: [ r\vphi'(r) ]'  \int_0^r ds \: s^{-1} n(h;U_s) \nonumber\\
&\overset{(\star)}{=}& \int_0^1 ds \: s^{-1} n(h ;U_s) \int_s^1 dr \:[ r\vphi'(r) ]'  = - \int_0^1 ds \: \vphi'(s)  n(h; U_s) \nonumber \\
&=& \int_0^\infty dt \:  \left[ \frac d {dt} \vphi(e^{-t}) \right]
n(h; U_{e^{-t}}). \label{eq:lem_compl1_id}
 \end{eqnarray}
The application of Fubini's theorem in ($\star$) is allowed by assumption (\ref{eq:rv1}). We recall the layer cake representation,
see Lieb and Loss \cite[Theorem 1.13]{Lieb}.
\begin{lemma}\label{lem:layer_cake}
  Let $\nu$ be a Borel measure on $\mr_+$ such that
  \begin{equation*}
    \Psi(t)=\nu([0,t))
  \end{equation*}
is finite for every $t>0$. Then for any Borel measure $\mu$ on $\mc$ and any $\mu$-measureable nonnegative function $f$
\begin{equation*}
  \int_\mc \Psi(f(z)) \mu(dz) = \int_0^\infty \: \mu(\{ z : f(z) > t\}) \nu(dt).
\end{equation*}
\end{lemma}
Applying the layer cake representation to the point measure
\begin{equation*}
  \mu_h(\{z\}) = \left\lbrace
    \begin{array}{cl}
      m(h;z) &, h(z)=0, \;z \in U \\
      0 &, \text{else}
    \end{array}\right.
\end{equation*}
where $m(h;z)$ counts the multiplicity of a zero $z$ of $h$, and to the measure
\begin{equation*}
  \nu_\vphi(dt)= \left[ \frac d {dt} \vphi(e^{-t}) \right] dt
\end{equation*}
we can reformulate the RHS of (\ref{eq:lem_compl1_id}) as follows
\begin{eqnarray}
&& \int_0^\infty dt \: \left[ \frac d {dt} \vphi(e^{-t}) \right] n(h; U_{e^{-t}})\nonumber
\\ &=& \int_0^\infty \nu_\vphi(dt) \:  \mu_h(\{ z :  -\log|z| > t \}) = \int_\mc \nu_\vphi([0,-\log|z|)) \mu_h(dz) \nonumber \\
&=& \sum_{z \in U,\; h(z)=0} \int_0^{-\log|z|} \: dt \: \left[ \frac
d {dt} \vphi(e^{-t}) \right] = \sum_{z \in U,\; h(z)=0} \vphi(|z|).
\nonumber
\end{eqnarray}
Now (\ref{eq:lem_compl1_id}) yields the result.
\end{proof}
In order to find conditions on $h$ and $\vphi$ that ensure the
RHS of (\ref{eq:compl1}) to be finite, we introduce the
following class of functions.
\begin{definition}
  Let $M(E,\alpha,\beta)$ denote the set of all functions $m : U \to \mr_+$ that obey an estimate of the form
  \begin{equation}\label{eq:m}
    m(z) \leq \frac{C_0}{(1-|z|)^\alpha \dist(z,E)^\beta }
  \end{equation}
where $C_0 > 0$ and $E \subset \partial U$ is any finite subset of the unit circle.
\end{definition}
We present a result on the finiteness of the RHS of
(\ref{eq:compl1}) in case that $\log |h(z)| \in M(E,\alpha,\beta)$
for some $\alpha,\beta\geq 0$. We do not try to present the most
general result in terms of the function $\vphi$ but will restrict
ourselves to one particular choice, namely $\vphi(r)=(1-r)^\gamma,
\gamma > 1$.
\begin{lemma}\label{lem:com2}
Let $m\in M(E,\alpha,\beta)$ for some $\alpha,\beta \geq 0$ and some finite $E\subset \partial U$. Let $h
: U \to \mc$ be analytic with $h(0)=1$ and
\begin{equation}\label{eq:dew}
  |h(z)| \leq \exp(m(z)).
\end{equation}
Then for every $\gamma > \max(1+\alpha, \alpha + \beta)$
\begin{equation}\label{eq:c2}
\sum_{z \in U,\; h(z)=0} (1-|z|)^\gamma \leq C_{\gamma}(m)
\end{equation}
where each zero of $h$ is counted according to its multiplicity and
\begin{equation}\label{eq:cg1}
C_\gamma(m) =  \frac{\gamma}{2\pi} \int_{\frac 1 {{\gamma}}}^1 dr
\int_0^{2\pi} d\theta \:  \frac{(r\gamma-1)}{(1-r)^{2-\gamma}}
m(re^{i\theta})
\end{equation}
is a finite constant.
\end{lemma}
\begin{rem}\label{rem:rem}
  In \cite{borichev} it has been shown that condition (\ref{eq:dew}) implies that for every $\eps>0$
  \begin{equation}\label{fd}
   \sum_{z \in U,\; h(z)=0} (1-|z|)^{\alpha+1+\eps} \dist(z,E)^{(\beta-1+\eps)_+} \leq C,
  \end{equation}
where $C$ depends on $\alpha, \beta, \eps$ and $m$ in a way that is
not made explicit. In case that $\beta < 1$ the finiteness of the
LHS of (\ref{fd}) for every $\eps>0$ is equivalent to the finiteness of the LHS of
(\ref{eq:c2}) for every $\gamma > 1+\alpha$, whereas in case that $\beta \geq 1$ the finiteness of
(\ref{fd}) implies the finiteness of (\ref{eq:c2}), but not vice
versa. In this respect, the result obtained in \cite{borichev} is
stronger than ours.
\end{rem}

\begin{proof}
For $\gamma > 1$ let $\vphi(r)=(1-r)^\gamma$. Since
\[ [ r\vphi'(r) ]'  = \gamma  (1-r)^{\gamma-2} (r\gamma -1)\] we obtain from (\ref{eq:compl1})
and our assumptions, using the non-negativity of $\int_0^{2\pi}\log
|h(re^{i\theta})| d\theta$, which follows from (\ref{eq:Jensen}),
 \begin{eqnarray}
   \sum_{z \in U,\; h(z)=0} (1-|z|)^\gamma &=& \frac {\gamma} {2\pi} \int_0^1 dr \frac{(r\gamma-1)}{(1-r)^{2-\gamma}} \int_0^{2\pi} d\theta \log |h(re^{i\theta})| \nonumber \\
&\leq& \frac {\gamma}{2\pi} \int_{1/{\gamma}}^1 dr \frac{(r\gamma-1)}{(1-r)^{2-\gamma}}   \int_0^{2\pi} {d\theta} \: m(re^{i\theta}) \nonumber \\
&\leq& \frac {C_0\gamma(\gamma-1)}
{2\pi} \int_{1/{\gamma}}^1 dr \frac{1}{(1-r)^{2-\gamma+\alpha}}
\int_0^{2\pi} \frac{d\theta}{\dist(re^{i\theta},E)^\beta}. \label{eq:compl2_int}
 \end{eqnarray}
It remains to show that the integral on the RHS of
(\ref{eq:compl2_int}) is finite whenever $\gamma >
\max(1+\alpha,\alpha +\beta)$. To this end we denote $E=\{
e^{i\theta_1}, \ldots , e^{i\theta_n} \}$ where $0 \leq \theta_1 <
\ldots < \theta_n < 2\pi$. Let $$\delta=\frac 1 4 \min_{1\leq k\leq
n} |e^{i\theta_{k+1}} - e^{i\theta_k}| \quad
,\theta_{n+1}:=\theta_1.$$ We further define
\begin{eqnarray*}
  G_k= \{ \theta \in [0,2\pi]: |e^{i\theta}-e^{i\theta_k}|<\delta\}, &  k=1,\ldots,n.
\end{eqnarray*}
Since for $r \geq 0$
$$\sup_{\theta \notin \cup_k G_k} \dist(re^{i\theta},E) \geq C>0, $$ the integral
\begin{eqnarray*}
   \int_{1/{\gamma} }^1 dr  \frac{1}{(1-r)^{2-\gamma+\alpha}} \int_{ \theta \notin \cup_k G_k} \frac{d\theta}{\dist(re^{i\theta},E)^\beta}
\end{eqnarray*}
is finite whenever $\gamma > \alpha +1$. It remains to show the
finiteness of
\begin{eqnarray}\label{eq:compl2_g}
   \int_{1/{\gamma} }^1 dr \frac{1}{(1-r)^{2-\gamma+\alpha}}  \int_{\cup_k G_k} \frac{d\theta}{\dist(re^{i\theta},E)^\beta}.
\end{eqnarray}
But for $\theta \in G_k$
\begin{equation*}
  \dist(re^{i\theta},E)= |re^{i\theta}-e^{i \theta_k}|
\end{equation*}
and hence
\begin{eqnarray}
  \int_{\cup_k G_k} \frac{d\theta}{\dist(re^{i\theta},E)^\beta} &=&
  \sum_k \int_{G_k} \frac{d\theta}{|re^{i\theta}-e^{i\theta_k}|^\beta}. \label{eq:compl2_r1}
\end{eqnarray}
It is not difficult to show that as $r \to 1$
\begin{eqnarray}\label{eq:compl2_r2}
  \int_{G_k} \frac{d\theta}{|re^{i\theta}-e^{i\theta_k}|^\beta} = \left\{
    \begin{array}{cl}
      O\left( \frac 1 {(1-r)^{\beta-1}} \right), & \beta > 1 \\[4pt]
      O\left( -\log(1-r) \right), & \beta = 1 \\[4pt]
      O\left(1 \right), & \beta < 1. \\[4pt]
    \end{array}\right.
\end{eqnarray}
We skip the elementary but technical calculation. Estimates
(\ref{eq:compl2_r1}) and (\ref{eq:compl2_r2}) show that the integral
in (\ref{eq:compl2_g}) is finite whenever $\gamma >
\max(1+\alpha,\alpha + \beta)$.
\end{proof}
\begin{rem}
In case that $m \in M(E,0,\beta)$ for $\beta < 1$, the above proof
actually shows that $h$ is element of the Nevanlinna class, i.e.
\[ \sup_{0<r<1} \int_0^{2\pi} \log_+|h(re^{i\theta})| \: d\theta < \infty.\]
As is well-known this implies the stronger result that $\sum_{z \in
U,\; h(z)=0} (1-|z|) < \infty$, see e.g. Rudin \cite[page 311]{Rudin}.
\end{rem}
We conclude this section with the classical Koebe distortion
theorem, which will be used later.
\begin{thm}\label{thm:koebe}
  Let $f : U \to \mc$ be conformal. Then
  \begin{equation*}\label{lem:koebe}
    \frac 1 4 |f'(z)| (1-|z|^2) \leq \dist(f(z),\partial f(U)) \leq |f'(z)| (1-|z|^2), \quad z \in U.
  \end{equation*}
\end{thm}
For a proof we refer to Pommerenke \cite[Cor. 1.4]{Pommerenke}.
\section{Eigenvalue estimates for bounded operators}\label{sec:bounded}
Let $A, B \in \mathbf{B}(\mathcal{H})$. Assume that $A$ is selfadjoint with
\begin{equation*}
\sigma(A) = [a,b], \quad a < b
\end{equation*}
and
\begin{equation*}
B-A \in \mathbf{S}_p \quad \text{for some } p > 0.
\end{equation*}
The last assumption and Remark \ref{rem:dav} imply that
\begin{equation*}
\sigma(B) \cap (\mc \setminus [a,b]) = \sigma_{disc}(B).
\end{equation*}
To obtain information on $\sigma_{disc}(B)$ we define a  conformal map $ k: \hat{\mc} \setminus [a,b] \to U $ as follows:
\begin{equation}\label{eq:ff}
k = k_{a,b}= w^{-1} \circ g,
\end{equation}
where $g :  \hat{\mc} \setminus [a,b] \to \hat{\mc} \setminus [-1,1] $ with
\begin{equation*}
 g(z)= \frac 1 {b-a} \left(2z-(b+a) \right), \quad   g^{-1}(z)= \frac 1 2 \left( (b-a)z + (b+a) \right)
\end{equation*}
and $w: U \to \hat{\mc} \setminus [-1,1]$ with
\begin{equation}\label{eq:ww}
 w(z)=\frac{1}{2} \left( z + z^{-1} \right), \quad w^{-1}(z)=  z - \sqrt{ z^2-1}.
\end{equation}
With $h_{A,B}^{(p)}$ as defined in (\ref{eq:per_det}) the composition
\begin{equation*}
[h_{A,B}^{(p)} \circ k^{-1}](z)= {\det}_{\lceil p
\rceil}(I-(k^{-1}(z)-A)^{-1}(B-A))
\end{equation*}
is analytic on $U$, by (\ref{eq:infty}) we have $[h_{A,B}^{(p)}\circ
k^{-1}](0)=1$, and $z_0$ is a zero of this function if and only if
$k^{-1}(z_0)$ is an eigenvalue of $B$ of the same multiplicity.
Since
\begin{small}
\begin{equation*}
  |[h_{A,B}^{(p)} \circ k^{-1}](z)| \leq \exp\left( \Gamma_p \|(k^{-1}(z)-A)^{-1}(B-A)\|_{\mathbf{S}_p}^p \right)
\end{equation*}
\end{small}
by estimate (\ref{eq:estimate_dunford}), the following theorem is a direct consequence of Lemma \ref{lem:com2}.
\begin{thm}\label{thm:te1}
  Let $m \in M(E,\alpha,\beta)$ for some finite $E\subset \partial U$, $\alpha,\beta \geq 0$ and suppose that for $z \in U$
  \begin{equation}\label{eq:esta1}
    \|(k^{-1}(z)-A)^{-1}(B-A)\|_{\mathbf{S}_p}^p \leq m(z).
  \end{equation}
Then for $\gamma > \max(1+\alpha, \alpha + \beta)$
\begin{equation}\label{eq:com2}
\sum_{\lambda \in \sigma_{disc}(B)} (1-|k(\lambda)|)^\gamma \leq \Gamma_p C_{\gamma}(m)
\end{equation}
where the finite constant $C_\gamma(m)$ was defined in (\ref{eq:cg1}).
\end{thm}
\begin{rem}
In the summation on the LHS of (\ref{eq:com2}), each eigenvalue of
$B$ is counted according to its algebraic multiplicity. In the
following results in this paper, this will be taken for granted
whenever a sum involving eigenvalues is considered.
\end{rem}
If no further information on the operators $A$ and $B$ is available,
the obvious way to show the validity of (\ref{eq:esta1}) is to use the estimate
\begin{equation}\label{eq:wq1}
  \|(k^{-1}(z)-A)^{-1}(B-A)\|_{\mathbf{S}_p}^p \leq \|(k^{-1}(z)-A)^{-1}\|^p \|B-A\|_{\mathbf{S}_p}^p
\end{equation}
and the identity (here we use the assumption that $A$ is
selfadjoint)
\begin{equation}\label{eq:wq2}
\|(k^{-1}(z)-A)^{-1}\| = \frac 1 {\dist(k^{-1}(z),[a,b])}= \frac{2}{b-a} \frac{1}{\dist(w(z),[-1,1])}.
\end{equation}
The proof of the following Lemma is provided in the appendix.
\begin{lemma}\label{lem:ko1}
Let $w(z)$ be defined by (\ref{eq:ww}). For $z \in U$, we have
\begin{small}
  \begin{equation*}\label{eq:ko1}
 \frac{1}{4} \frac{|z^2-1|(1-|z|)}{|z|} \leq \dist(w(z),[-1,1]) \leq   \frac{1+\sqrt{2}}{4} \frac{|z^2-1|(1-|z|)}{|z|}.
  \end{equation*}
\end{small}
\end{lemma}
\begin{thm}\label{thm:kd1}
  For $\gamma > \max(1+p,2p)$ and $k=k_{a,b}$ as above we have
  \begin{equation*}\label{eq:kd1}
    \sum_{\lambda \in \sigma_{disc}(B)} (1-|k(\lambda)|)^\gamma \leq
    \Gamma_p \left( \frac{2}{b-a} \right)^p C_{\gamma,p} \| B-A
    \|_{\mathbf{S}_p}^p,
  \end{equation*}
where
\begin{equation}\label{eq:aa1}
C_{\gamma,p} =  \frac{\gamma}{2\pi} \int_{\frac 1 {{\gamma}}}^1 dr \frac{(r\gamma-1)}{(1-r)^{2-\gamma}}  \int_0^{2\pi} d\theta \:  \frac{1}{\dist(w(re^{i\theta}),[-1,1])^p}
\end{equation}
is a finite constant.
\end{thm}
\begin{proof}
From (\ref{eq:wq1}) and (\ref{eq:wq2}) we obtain
\begin{equation*}\label{eq:fd}
   \|(k^{-1}(z)-A)^{-1}(B-A)\|_{\mathbf{S}_p}^p \leq  \left( \frac{2}{b-a} \right)^p \| B-A \|_{\mathbf{S}_p}^p \frac{1}{\dist(w(z),[-1,1])^p}.
\end{equation*}
Since $ z \mapsto (\dist(w(z),[-1,1]))^{-p} \in M(\{-1,1\},1+p,2p)$
by Lemma \ref{lem:ko1}, we obtain from Theorem \ref{thm:te1} that
for $\gamma > \max(1+p,2p)$
\begin{equation*}
  \sum_{\lambda \in \sigma_{disc}(B)} (1-|k(\lambda)|)^\gamma \leq \Gamma_p \left( \frac{2}{b-a} \right)^p   C_{\gamma,p} \| B-A \|_{\mathbf{S}_p}^p.
\end{equation*}
Here the finite constant $C_{\gamma,p}$ is defined as in (\ref{eq:aa1}).
\end{proof}
Lemma \ref{lem:ko1} can be used to obtain a more transparent
formulation of Theorem \ref{thm:kd1}.
\begin{thm}\label{thm:aa2}
  Let $\gamma > \max(1+p,2p)$. Then
\begin{equation*}
 \sum_{\lambda \in \sigma_{disc}(B)} \frac{\dist(\lambda,[a,b])^\gamma}{|\lambda-a|^{\frac \gamma 2}|\lambda-b|^{\frac \gamma 2 }}   \leq  \Gamma_p \left( \frac{2}{b-a} \right)^{p} \left(\frac{1+\sqrt{2}}{2} \right)^\gamma C_{\gamma,p} \| B-A \|_{\mathbf{S}_p}^p
\end{equation*}
where the finite constant $C_{\gamma,p}$ was defined in (\ref{eq:aa1}).
\end{thm}
\begin{proof}
  From Lemma \ref{lem:ko1} we get for $z = k(\lambda)=w^{-1}(g(\lambda))$
  \begin{eqnarray*}
    \dist(\lambda , [a,b]) &=& \dist(g^{-1}(w(z)),[a,b])\\
   &=& \frac{b-a}{2} \dist(w(z),[-1,1]) \leq  \frac{b-a}{2} \frac{1+\sqrt{2}}{4} \frac{|z^2-1|}{|z|}(1-|z|) \\
&=& \frac{b-a}{2}  \frac{1+\sqrt{2}}{4}  \frac{|k(\lambda)^2-1|}{|k(\lambda)|} (1-|k(\lambda)|)  \\
&=& \frac{1+\sqrt{2}}{2} {|(\lambda-a)(\lambda-b)|^{1/2}}
(1-|k(\lambda)|),
  \end{eqnarray*}
  so that
$$(1-|k(\lambda)|)^\gamma \geq
\Big(\frac{2}{1+\sqrt{2}}\frac{\dist(\lambda,[a,b])}{|(\lambda-a)(\lambda-b)|^{\frac{1}{2}}}\Big)^\gamma,$$
and an application of Theorem \ref{thm:kd1} concludes the proof.
\end{proof}
\section{Eigenvalue estimates for unbounded operators}\label{sec:unboundedI}
Let $H_0, H  \in \mathbf{C}(\mathcal{H})$ and suppose that $H_0$ is selfadjoint with $\sigma(H_0)=[0,\infty)$. To apply the results of the last section, we assume that
\begin{equation}
  \label{eq:assumption_res}
  R_s=(s-H)^{-1}-(s-H_0)^{-1} \in \mathbf{S}_p
\end{equation}
for some $p > 0$ and $s \in \rho(H_0) \cap \rho(H)\cap \mr_-$. The last assumption, together with the spectral mapping theorem for resolvents, implies that
\begin{equation*}
    \sigma(H) \cap (\mc \setminus [0,\infty)) = \sigma_{disc}(H).
  \end{equation*}
  \begin{rem}
Given assumption $(\ref{eq:assumption_res})$ there might exist a
sequence of eigenvalues of $H$ that diverges to infinity, i.e. the
points of $\sigma_{disc}(H)$ can accumulate in $[0,\infty) \cup
\{\infty\}$. However, $(\ref{eq:assumption_res})$ implies some
restrictions on the rate of divergence as can be seen from the next
theorem.
  \end{rem}
\begin{thm}\label{thm:t3}
Let $H_0, H$ be as above and let $\gamma > \max(1+p,2p)$. Then
  \begin{equation*}
\sum_{\lambda \in \sigma_{disc}(H)} \frac{ \dist(\lambda,[0,\infty))^\gamma }{|\lambda|^{\frac \gamma 2}
(1+|\lambda|)^{\gamma}}  \leq  2^{p}C_s^\gamma |s|^{(\frac{\gamma}{2}+p)}
\Gamma_p C_{\gamma,p}\|R_s\|_{\mathbf{S}_p}^p,
  \end{equation*}
where the finite constant $C_{\gamma,p}$ was defined in  (\ref{eq:aa1}) and
\begin{equation}\label{eq:vcx}
  C_s = 4 \sup_{z \in U} \frac{1+|z|}{|z-1|^2+|s| |z+1|^2}.
\end{equation}
\end{thm}
For the proof of this theorem we will need the contents of the next lemma.
\begin{lemma}\label{lem:hg}
Let $ s \in \mr_-$ and define
\begin{equation}\label{eq:kh2}
l_s : \mc \setminus [0,\infty) \to U, \quad
l_s(\lambda)=k_{s^{-1},0}((s-\lambda)^{-1}),
\end{equation} where $k_{s^{-1},0}$ was defined in (\ref{eq:ff}). Then the following holds for $\lambda \in \mc \setminus [0,\infty)$
\begin{small}
\begin{equation*}
\frac 1 4 \left| \frac{\lambda}{s} \right|^{1/2}
\frac{|\lambda-s|(1-|l_s(\lambda)|^2)}{|l_s(\lambda)|} \leq
\dist(\lambda,[0,\infty)) \leq \left| \frac{\lambda}{s}
\right|^{1/2}
\frac{|\lambda-s|(1-|l_s(\lambda)|^2)}{|l_s(\lambda)|}.
 \end{equation*}
\end{small}
\end{lemma}
\begin{proof}
  We note that by definition
  \begin{equation*}
    l_s(\lambda)=  \left( \frac{\lambda + s }{\lambda -s} \right)  - \sqrt{ \left( \frac{\lambda + s }{\lambda -s} \right)^2-1}, \quad
l_s^{-1}(z)= s \left( \frac{z+1}{z-1} \right)^2
  \end{equation*}
and $l_s^{-1}$ is  a conformal map of $U$ onto $\mc \setminus
[0,\infty)$. For $z=l_s(\lambda)$ we can thus use the Koebe theorem
(Theorem \ref{thm:koebe}) to obtain
\begin{eqnarray*}
  \dist(\lambda,[0,\infty)) &=& \dist( l_s^{-1}(z), \partial l_s^{-1}(U)) \leq |[l_s^{-1}]'(z)|(1-|z|^2) \nonumber \\
&=& 4 |s| \left| \frac{z^2-1}{(z-1)^4} \right| (1-|z|^2)= 4 |s| \left| \frac{l_s(\lambda)^2-1}{(l_s(\lambda)-1)^4} \right| (1-|l_s(\lambda)|^2) \nonumber \\
 &=& \left| \frac{\lambda}{s} \right|^{1/2} \frac{|\lambda-s|(1-|l_s(\lambda)|^2)}{|l_s(\lambda)|}.
\end{eqnarray*}
Here the last equality follows by some algebraic manipulations. The lower bound on $\dist(\lambda,[0,\infty))$ is obtained in exactly the same manner.
\end{proof}
\begin{proof}[Proof of Theorem \ref{thm:t3}]
Let $\gamma > \max(1+p,2p)$. Since $\sigma((s-H_0)^{-1})=[\frac 1 s,0]$ we can apply Theorem \ref{thm:kd1} to the bounded operators $A=(s-H_0)^{-1}$ and $B=(s-H)^{-1}$ to obtain
 \begin{equation}\label{eq:gf2}
    \sum_{\mu \in \sigma_{disc}((s-H)^{-1})} (1-|k_{s^{-1},0}(\mu)|)^\gamma \leq   \Gamma_p 2^p |s|^p  C_{\gamma,p} \| R_s \|_{\mathbf{S}_p}^p
  \end{equation}
where $k_{s^{-1},0}$ and $C_{\gamma,p}$ were defined in (\ref{eq:ff}) and (\ref{eq:aa1}) respectively.
Since $\mu \in \sigma_{disc}((s-H)^{-1})$ if and only if $s- \frac{1}{\mu} \in \sigma_{disc}(H)$ we can reformulate the LHS of (\ref{eq:gf2}) as follows
\begin{eqnarray}\label{eq:p1}
\sum_{\mu \in \sigma_{disc}((s-H)^{-1})}
(1-|k_{s^{-1},0}(\mu)|)^\gamma =
 \sum_{\lambda \in \sigma_{disc}(H)} (1-|l_s(\lambda)|)^\gamma,
\end{eqnarray}
where   the function $l_s$ was defined in (\ref{eq:kh2}). From Lemma
\ref{lem:hg} we have
\begin{equation}
   1-|l_s(\lambda)|\geq
   \left[\frac{|l_s(\lambda)|(1+|\lambda|)}{|\lambda-s|(1+|l_{s}(\lambda)|)}\right]
   \left[\frac{|s|^{\frac{1}{2}}\dist(\lambda,[0,\infty))}{|\lambda|^{\frac{1}{2}}(1+|\lambda|)}\right]. \label{eq:fa}
\end{equation}
Furthermore, a short computation shows that
\begin{eqnarray*}
 \inf_{\lambda \in \mc \setminus [0,\infty)} \frac{ |l_s(\lambda)|(1+|\lambda|)}{|\lambda-s|(1+|l_s(\lambda)|)}=
  \inf_{z \in U} \frac{ |z|(1+|l_s^{-1}(z)|)}{|l_s^{-1}(z)-s|(1+|z|)}
  = \frac{1}{|s|C_s}
\end{eqnarray*}
where $C_s \in (0,\infty)$ was defined in (\ref{eq:vcx}). From (\ref{eq:fa}) we thus  obtain
\begin{equation*}
1-|l_s(\lambda)| \geq   \frac{1}{|s|^{1/2}C_s}
\frac{\dist(\lambda,[0,\infty))}{|\lambda|^{1/2}(1+|\lambda|)}.
\end{equation*}
With (\ref{eq:gf2}) and (\ref{eq:p1}) this concludes the proof of the theorem.
\end{proof}
\section{A counterexample}

In this section we present a counterexample which shows that, in one
respect, Theorem \ref{thm:kd1} and Theorem \ref{thm:aa2} are optimal:
For $\gamma < 1$ it is not possible to conclude the finiteness of \[ \sum_{\lambda \in \sigma_{disc}(B)} \frac{\dist(\lambda,[a,b])^\gamma}{|\lambda-a|^{\gamma/2}|\lambda-b|^{\gamma/2}}\] in terms of Schatten class properties of $B-A$.\\[4pt]
We work on the space $l^2(\mz)$, and denote its natural basis
by $\{\delta_j\}_{j\in\mz}$, where $\delta_j$ is defined by
$\delta_j(j)=1$, $\delta_j(k)=0$ for $k\neq j$.\\[4pt]
We define $A: l^2(\mz)\rightarrow l^2(\mz)$ to be the
discrete free Schr\"odinger operator:
\[ (A u)(k)=u(k-1)+u(k+1),\quad u\in l^2(\mz), \quad k \in \mz. \]
The spectrum of $A$ is $[-2,2]$.

\begin{prop}\label{fl}
Given any sequence $\{ \lambda_k\}_{k \in \mn} \subset \mc\setminus
[-2,2]$ which satisfies
\begin{equation*}\label{cond}\sum_{k} \frac{\dist(\lambda_k,[-2,2])}
{|\lambda_k+2|^{1/2}|\lambda_k-2|^{1/2}}<\infty,\end{equation*} there exists a
rank-one operator $M$ such that, setting $B=A+M$, we have
$\{ \lambda_k\} \subset \sigma_{disc}(B)$.
\end{prop}

Since we may choose $\lambda_k$ in Lemma \ref{fl} to be, e.g.,
$\lambda_k=k^{-(1+\delta)}i$, with
$\delta>0$ arbitrarily small, we immediately get
\begin{prop}
For any $\gamma<1$, there exists a rank-one operator $M$ such that
the eigenvalues of $B=A+M$ satisfy
$$\sum_{\lambda\in \sigma_{disc}(B)} \frac{\dist(\lambda,[-2,2])^\gamma}
{|\lambda+2|^{\frac{\gamma}{2}}|\lambda-2|^{\frac{\gamma}{2}}}=+\infty.$$
\end{prop}
Since a rank-one perturbation belongs to all Schatten classes $\mathbf{S}_p$,
$p>0$, this shows that there is no hope to obtain the results of
Theorem \ref{thm:kd1} and Theorem \ref{thm:aa2} for $\gamma<1$, under an assumption of
the form $B-A\in \mathbf{S}_p$.\\[4pt]
We also note that the above propositions demonstrate a striking
difference between the selfadjoint and non-selfadjoint cases. Given
a selfadjoint operator with no eigenvalues, it is well-known that a
selfadjoint rank-one perturbation of this operator can have at most
one eigenvalue outside of its essential spectrum. Here we see that a
non-selfadjoint rank-one perturbation of a selfadjoint operator can
give birth to infinitely many eigenvalues.

\begin{proof}[Proof of Proposition \ref{fl}]
The rank-one perturbation $M$ is defined by:
\[ Mu=\Big[\sum_{j=-\infty}^\infty \alpha_j u(j) \Big]\delta_0, \quad u \in l^2(\mz),\]
where $\alpha_j$ are
to be determined below. For $M$ to be bounded, we need to
assume that
\begin{equation}\label{as}
\sum_{j=-\infty}^\infty |\alpha_j|^2<\infty.
\end{equation}
We now look for eigenvectors $u_z\in l^2(\mz)$ of $B=A+M$ of the
form
$$u_z(k)=z^{|k|},$$
with $|z|<1$. Note that
\begin{equation}\label{c1}|k|\geq 1\;\Rightarrow\;(B
u_z)(k)=z^{|k|}(z^{-1}+z)\end{equation}
\begin{equation}\label{c3}(Bu_z)(0)=2z+\sum_{j=-\infty}^\infty \alpha_j z^{|j|}=
\alpha_0+(\alpha_1+\alpha_{-1}+2)z+\sum_{j=2}^\infty
(\alpha_j+\alpha_{-j}) z^{j}.\end{equation} By (\ref{c1}), we see
that if $u_z$ is an eigenvector then the corresponding eigenvalue is
$\lambda=z+z^{-1}$. From (\ref{c3}) we then get that a necessary and
sufficient condition for $u_z$ to be an eigenvector is that
$$\alpha_0+(\alpha_1+\alpha_{-1}+2)z+\sum_{j=2}^\infty (\alpha_j+\alpha_{-j}) z^{j}=\lambda =z+z^{-1},$$
which we can write as $\phi(z)=0$ where $\phi(z)$ is defined by
\begin{equation}\label{eq:dd}
\phi(z)=-1+\alpha_0 z+(\alpha_1+\alpha_{-1}+1)z^2+\sum_{j=3}^\infty (\alpha_{j-1}+\alpha_{-j+1}) z^{j}.
\end{equation}
Thus the numbers of the form $\lambda=z+z^{-1}$, where $z$ are the
zeroes of $\phi$ in $U$, are
eigenvalues of $B$. Note that by assumption (\ref{as}),
$\phi(z)\in H^2(U)$.\\[4pt]
Let $\{ \lambda_k \} \subset \mc \setminus [-2,2]$ be any sequence that satisfies
\begin{equation}\label{eq:tz}
  \sum_{k=1}^\infty \frac{ \dist(\lambda_k, [-2,2])}{ |\lambda_k^2-4|^{1/2}} < \infty.
\end{equation}
In the following, we will select a specific sequence $\{\alpha_j\}$ such that $\{\lambda_k\} \subset \sigma_{disc}(B)$, where $B=B(\{ \alpha_j\})$ as defined above. To this end, we define the sequence $\{ z_k \}  \subset U \setminus \{0\}$ by
\[ \lambda_k = z_k + z_k^{-1}.\]
As in the proof of Theorem \ref{thm:aa2} one can use Lemma \ref{lem:ko1} to  check that condition (\ref{eq:tz}) on $\lambda_k$ is equivalent to
\begin{equation}\label{fs}\sum_{k=1}^\infty(1-|z_k|)<\infty.\end{equation}
By a well-known result from complex analysis, see e.g. Rudin
\cite[page 310]{Rudin}, (\ref{fs}) implies that one can construct a
function $g\in H^2(U)$ (in fact even $g\in
H^\infty(U)$) whose zeroes are $\{z_k\}$.\\[4pt]
We can normalize $g$ so that
$g(0)=-1$. Denoting the Taylor expansion of $g$ by
$$g(z)=-1+\sum_{j=1}^\infty \beta_j z^j,$$
we can choose $\alpha_0=\beta_1$, $\alpha_1=\beta_2-1$,
$\alpha_j=\beta_{j+1}$ for $j\geq 2$ and $\alpha_j=0$ for $j<0$, so
that from (\ref{eq:dd}) we obtain $\phi=g$. From the considerations above, this implies that
$\lambda_k=z_k+z_k^{-1}$ are eigenvalues of $B$. We  have thus
proven Proposition \ref{fl}.
\end{proof}

\section{Appendix}
\begin{proof}[Proof of Lemma \ref{lem:ko1}]
For $w=\frac 1 2 \left( z + z^{-1} \right)$ we define
\[ Z_1=\{ z : \Re w \leq -1\}, \quad Z_2= \{ z : \Re w \geq 1\}, \quad Z_3= \{ z : |\Re w| < 1\}\]
where $\Re w = \frac {\Re z}{2} \left( \frac{1+|z|^2}{|z|^2} \right)$.
Then
\begin{eqnarray}
  \dist(w,[-1,1]) &=& \left\{
    \begin{array}{cl}
      |w+1|= \frac 1 2  \frac{|1+z|^2}{|z|} & , z \in Z_1 \\[2pt]
      |w-1|= \frac 1 2  \frac{|1-z|^2}{|z|} & , z \in Z_2 \\[2pt]
      |\Im w|= \frac{|\Im z|}{2} \frac{1-|z|^2}{|z|^2} & , z \in Z_3.
    \end{array}\right. \label{eq:di}
\end{eqnarray}
We first show that for $z \in Z_3$ the following holds
\begin{equation}\label{eq:ay}
  \frac{\sqrt{2}}{4} \frac{|z^2-1|(1-|z|)}{|z|} \leq \dist(w,[-1,1]) \leq \frac{1+\sqrt{2}}{4} \frac{|z^2-1|(1-|z|)}{|z|}.
\end{equation}
With (\ref{eq:di}) this is equivalent to
\begin{equation}\label{eq:aax}
\frac 1 {\sqrt{2}} \leq   {|\Im z|} \frac{1+|z|}{|z||z^2-1|} \leq \frac{1+\sqrt{2}}{2}.
\end{equation}
Switching to polar coordinates we see that  $re^{i \theta} \in Z_3$ if $\cos^2(\theta) < 4 \frac {r^2} {(1+r^2)^2}$ and (\ref{eq:aax}) can be rewritten as follows
\begin{equation}\label{eq:as1}
\frac 1 {\sqrt{2}} \leq  \frac{(1+r)\sqrt{1-\cos^2(\theta)}}{\sqrt{(1+r^2)^2-4{r^2}\cos^2(\theta)}} \leq \frac{1+\sqrt{2}}{2}.
\end{equation}
For $x=\cos^2(\theta)$ and fixed $r$ we define
\begin{equation*}\label{eq:kl2}
f(x)= \frac{{1-x}}{{(1+r^2)^2-4{r^2}x}} \quad , 0 \leq x < 4 \frac {r^2} {(1+r^2)^2}.
\end{equation*}
It is easy to see that $f$ is monotonically decreasing. We thus obtain
\begin{equation*}\label{eq:ad}
\frac{1}{1+6r^2+r^4}= f\left( 4 \frac {r^2} {(1+r^2)^2} \right) \leq  f(x) \leq f(0) = \frac{1}{(1+r^2)^2}.
\end{equation*}
The last chain of inequalities implies  the validity of (\ref{eq:as1}) and (\ref{eq:aax}) since
\[ \sup_{r \in [0,1]} \frac{1+r}{1+r^2}=\frac{1+\sqrt{2}}{2} \quad \text{and} \quad \inf_{r \in [0,1]} \frac{1+r}{\sqrt{1+6r^2+r^4}}= \frac 1 {\sqrt{2}}. \]
Next, we show that for  $z \in Z_1 \cup Z_2$
\begin{equation*}\label{eq:az}
\frac 1 4 \frac{|z^2-1|(1-|z|)}{|z|} \leq \dist(w,[-1,1]) \leq \frac{1+\sqrt{2}}{4} \frac{|z^2-1|(1-|z|)}{|z|}.
\end{equation*}
By symmetry, it is sufficient to show it for $z \in Z_1$, i.e. we have to show
\begin{equation}\label{eq:rt}
  \frac 1 4 \frac{|z^2-1|(1-|z|)}{|z|} \leq \frac 1 2 \frac{|z+1|^2}{|z|} \leq \frac{1+\sqrt{2}}{4} \frac{|z^2-1|(1-|z|)}{|z|} \quad ,z \in Z_1.
\end{equation}
In polar coordinates this is equivalent to
\begin{equation}\label{eq:wk}
  \frac 1 2  \leq  \frac{1}{1-r} \sqrt{ \frac{r^2+1+2r\cos(\theta)}{r^2+1-2r\cos(\theta)}} \leq \frac{1+\sqrt{2}}{2}
\end{equation}
for $\cos(\theta) \leq -2 \frac {r} {1+r^2}$. For $y=\cos(\theta)$ and fixed $r$ we define
\begin{equation}\label{eq:kj}
  q(y)=\frac{r^2+1+2ry}{r^2+1-2ry} \quad, -1 \leq y \leq -2 \frac {r} {1+r^2}.
\end{equation}
A short calculation shows that $q$ is monotonically increasing and we obtain that
\begin{equation}\label{eq:kl}
  \left( \frac{ 1-r}{1+r} \right)^2= q(-1) \leq q(y) \leq q \left( -2 \frac {r} {1+r^2} \right) =  \frac{(1-r^2)^2}{1+6r^2+r^4} .
\end{equation}
(\ref{eq:kj}) and (\ref{eq:kl}) imply the validity of (\ref{eq:wk}) and (\ref{eq:rt}) since
\begin{equation}
  \inf_{r \in [0,1]} \frac{1}{1+r} = \frac 1 2 \quad \text{ and } \quad \sup_{r \in [0,1]}  \frac{1+r}{\sqrt{1+6r^2+r^4}} \leq \frac{1+\sqrt{2}}{2}.
\end{equation}
This concludes the proof of the lemma.
\end{proof}

\end{document}